\documentclass[11pt]{article}
\usepackage{amssymb}
\usepackage{amsfonts}
\usepackage{amsmath}
\usepackage{cases}
\usepackage[usenames]{color}
\usepackage{mathrsfs}
\usepackage[numbers,sort&compress]{natbib}
\usepackage{cases}
\usepackage{amsthm}

\usepackage{indentfirst}
\usepackage{setspace}
\usepackage{abstract}
\usepackage[none]{hyphenat}

\allowdisplaybreaks[4]
\def\ptl{\partial}

\def\cl{\centerline}

\def\al{\alpha}

\def\b{\beta}

\def\vs{\vspace*}

\def\Z{\mathbb{Z}}
\def\N{\mathbb{N}}

\def\C{\mathbb{C}}
\def\g{\mathbf{g}}

\def\C{\mathbb{C}}

\def\SUM#1#2{{\mbox{$\sum\limits_{#1}^{#2}$}}}

\newcommand{\re}[1]{\cite{#1}}  
\numberwithin{equation}{section}
\newtheorem{theo}{Theorem}[section]
\newtheorem{defi}[theo]{Definition}
\newtheorem{coro}[theo]{Corollary}
\newtheorem{lemm}[theo]{Lemma}

\newtheorem{prop}[theo]{Proposition}
\newtheorem{clai}{Claim}

\newtheorem{remark}[theo]{Remark}

\begin{document}
\sloppy{}
\baselineskip 6pt
\lineskip 6pt

\begin{center}{\Large\bf Irreducible tensor product modules over the \\ Takiff Lie algebra for $\mathfrak{sl}_{2}$}
\footnote {Corresponding author: Xiaoyu Zhu (zhuxiaoyu1@nbu.edu.cn)}
\end{center}
\vs{6pt}

\cl{Yu Qiao,\ \ Xiaoyu Zhu}
\cl{\footnotesize  School of Mathematics and Statistics, Ningbo University, Ningbo, Zhejiang, 315211, China}
\cl{\footnotesize E-mails: qiaoyu20240912@163.com, zhuxiaoyu1@nbu.edu.cn}
\quad\\
\vs{12pt} \par
\noindent{\bf Abstract.}
In this paper, we construct a class of non-weight modules over the Takiff $\mathfrak{sl}_{2}$ 
by taking the tensor products of the irreducible free $U(\overline{\mathfrak{h}})$-modules of rank 1,
where $\overline{\mathfrak{h}}$ is a natural Cartan subalgebra of the Takiff $\mathfrak{sl}_{2}$, 
with the irreducible highest weight modules. We characterize the irreducibility of these tensor product modules 
and determine the necessary and sufficient conditions for isomorphisms between them. 
We further prove that these non-weight modules are distinct from the known non-weight modules. 
Finally, we reformulate some tensor product modules over the Takiff $\mathfrak{sl}_{2}$ as induced modules derived from modules over certain subalgebras, 
and determine the necessary and sufficient conditions for the reducibility of these induced modules.

\noindent{\bf Keywords:} Takiff $\mathfrak{sl}_{2}$, non-weight modules, tensor product, highest weight modules.

\section{Introduction}

The representation theory of Lie algebras has attracted significant attention from mathematicians, 
encompassing both weight representations and non-weight representations.
In contrast to semi-simple Lie algebras, the representation theory, for non-semi-simple Lie algebras remains far less studied.

There are several natural families of non-semi-simple Lie algebras, 
which include current Lie algebras, conformal Galilei algebras,
Takiff Lie algebras, and others. Among these, current Lie algebras are the one,
which are widely investigated due to their connections with Kac-Moody algebras and quantum groups.
For this type Lie algebras, a complete classification of Harish-Chandra modules has been established, see \cite{L2018}.
The highest weight theory over truncated current Lie algebra was investigated in \cite{W2011}.

Takiff Lie algebras can be seen as the ``smallest''
non-semi-simple truncated current Lie algebras. 
The Takiff Lie algebra for $\mathfrak{sl}_{2}$ also falls into 
the category of conformal Galilei algebras, see e.g.\cite{LMZ2014}, 
and is defined as the semidirect product of $\mathfrak{sl}^{}_{2}$ with its adjoint representation. 
These particular Lie algebras were first introduced by Takiff in \cite{T},
whose primary motivation was the invariant theory for such Lie algebras. 
Representation theory of Takiff Lie algebras and, especially, questions
related to highest weight modules and BGG category $\mathcal{O}$ have 
recently attracted considerable attention, see \cite{C2023,MM2022,MS,Z2025} and the references therein.
More recently, the author classified a family of non-weight modules,
known as $U(\mathfrak{h})$-free modules, 
over Takiff $\mathfrak{sl}_{2}$ in \cite{Z2024}.
Such modules were initially introduced by
J. Nilsson in \cite{N2015} and, independently, 
by H. Tan and K. Zhao in \cite{TZ2018} for the simple Lie algebra $\mathfrak{sl}^{}_{n+1}$.
The modules in this family are free of finite rank when restricted to a
fixed Cartan subalgebra of the Lie algebra in question. 
Since then, these ideas were exploited and generalized to 
consider modules over various Lie algebras, see \cite{CC2015,CG2017,SYZ2024,TZ2015}.
Furthermore, some subsequent work on constructing new modules by taking tensor products
of $U(\mathfrak{h})$-free modules with known simple modules on many Lie algebras,
such as the Virasoro algebra \cite{TZ2013, TZ2016, LGW2020}, the affine-Virasoro algebra of type $A_{1}$ \cite{CY2022}, 
the mirror Heisenberg-Virasoro algebra \cite{GMZ2022} and others. 
The simplicity and isomorphism classes
of the tensor product modules over these algebras have also been determined.
It is well-known that this is an important and efficient way to construct new modules over an algebra.
Motivated by these,
the aim of this paper is to construct new non-weight modules over Takiff Lie algebra for $\mathfrak{sl}_{2}$ $\mathbf{g}$
by taking the tensor product of the irreducible modules 
defined in \cite{Z2024} with irreducible highest weight modules.

The paper is organized as follows. 
In Section 2, we introduce the basic setup that we work in and recall
some necessary definitions and preliminary results. 
In section 3, we prove the irreducibility of the tensor product module $V\left(\lambda,a,b_{\beta}\right)\otimes L\left(\eta,\theta\right)$,
where $V\left(\lambda,a,b_{\beta}\right)=\Gamma\left(\lambda,a,b\right)$, $\Theta(\lambda,a,b)$, or
$\Omega\left(\lambda,a,\beta\left(\overline{h}\right)\right)$ ($a\neq0$) (see Definition \ref{defi01}). 
Section 4 is devoted to classifying the isomorphism classes of the tensor product modules $V\left(\lambda,a,b_{\beta}\right)\otimes L\left(\eta,\theta\right)$. 
In section 5, we demonstrate that the tensor product modules
$V\left(\lambda,a,b_{\beta}\right)\otimes L\left(\eta,\theta\right)$ constructed in this paper are new 
by making a comparison with the other known non-weight modules. 
In the final section, we realize these tensor product modules as induced modules from certain subalgebras,
and give necessary and sufficient conditions for their reducibility.

\section{Preliminaries}\label{sec:preliminaries}
Throughout this paper, we denote by $\C$, $\C^{\times}$, $\Z$, $\Z_{+}$  and $\N$ the
sets of complex numbers, nonzero complex numbers, integers, non-negative integers and positive integers, respectively.
In this paper, all vector spaces (resp. algebras, modules) are defined over $\C$. We denote by $V^{*}$ the dual space of a vector space $V$. Denote by $\C[s,t]$ the polynomial algebra in variables $s$ and $t$ with coefficients in $\C$.

Consider the Lie algebra $\mathfrak{sl}_{2}$ with the standard basis $\{e, h, f\}$ and the Lie bracket
\begin{equation*}
  [e,f]=h,\ \ \ [h,e]=2e,\ \ \ [h,f]=-2f.
\end{equation*}
Let $D:=\C[t]/(t^{2})$ be the algebra of dual numbers. Consider the associated {\it Takiff Lie algebra} $\mathbf{g}:=\mathfrak{sl}_{2}\otimes_{\C}D$ with the Lie bracket
\begin{equation*}
  [a\otimes t^{i}, b\otimes t^{j}]=[a, b]\otimes t^{i+j},
\end{equation*}
where $a, b\in \mathfrak{sl}_{2}$ and $i, j\in \left\{0,1\right\}$, and the Lie bracket on the right hand side being the usual $\mathfrak{sl}_{2}$-Lie bracket. Set
\begin{equation*}
  \overline{e}:=e\otimes t, \ \ \overline{f}:=f\otimes t, \ \ \overline{h}:=h\otimes t.
\end{equation*}
Let $\overline{\mathfrak{n}}_{+}$ be the subalgebra of $\g$ generated by $e, \overline{e}$. Let $\overline{\mathfrak{h}}$ be the subalgebra of $\g$ generated by $h, \overline{h}$, and,  finally,  let $\overline{\mathfrak{n}}_{-}$ be the subalgebra of $\g$ generated by $f, \overline{f}$. Then we have the following triangular decomposition of $\g:$
\begin{equation*}
  \g=\overline{\mathfrak{n}}_{+}\oplus \overline{\mathfrak{h}}\oplus \overline{\mathfrak{n}}_{-}.
\end{equation*}
\noindent Let $\gamma\left(h,\overline{h}\right)\in \C\left[h,\overline{h}\,\right]$ and write $\gamma\left(h,\overline{h}\right)=\sum _{i=0}^{m}\sum_{j=0}^{n}c_{i,j}h^{i}\overline{h}^{j}$ with $m,n\in \Z_{+}$ and $c_{i,j}\in \C$. 
 We define the degree of $h$ and $\overline{h}$ of $\gamma\left(h,\overline{h}\right)$ by $m$ and $n$, respectively. Denote by $\mbox{deg}_{y}\gamma\left(h,\overline{h}\right)$ the degree of $y$ of $\gamma\left(h,\overline{h}\right)$ for $y\in \left\{h,\overline{h}\right\}$.

Now let us recall the definitions of the $\g$-modules $\Gamma\left(\lambda,a,b\right)$, $\Theta(\lambda,a,b)$, 
$\Omega\left(\lambda,a,\beta\left(\overline{h}\right)\right)$ and $L(\eta,\theta)$ concerned in this paper, and some basic properties of them.
As vector spaces, $\Gamma\left(\lambda,a,b\right)$, $\Theta(\lambda,a,b)$ and  
$\Omega\left(\lambda,a,\beta\left(\overline{h}\right)\right)$ coincide with $\C\left[\,h,\overline{h}\,\right]$.
\begin{defi}\label{defi01}\rm{(see \cite{Z2024})} Let $\lambda\in \C^{\times}$, $a,b\in \C$ and $\gamma\left(h,\overline{h}\right)\in \C\left[\,h,\overline{h}\,\right]$. 
We can define the $\mathbf{g}$-module structure on $\C\left[\,h,\overline{h}\,\right]$ as follows:
\begin{align*}
  \Gamma\left(\lambda,a,b\right):\ \ \ e\cdot \gamma\left(h,\overline{h}\right)&=-2\lambda \overline{\partial}\left(\gamma\left(h-2,\overline{h}\right)\right),\ \ \ \ \ \ \ \ \ \ \ \ \overline{e}\cdot \gamma\left(h,\overline{h}\right)=\lambda\gamma\left(h-2,\overline{h}\right), \\
  \ \ \ \overline{f}\cdot \gamma\left(h,\overline{h}\right)&=-\frac{1}{4\lambda}\left(\overline{h}^{2}+a\right)\gamma\left(h+2,\overline{h}\right), \ \ \ \ x\cdot \gamma\left(h,\overline{h}\right)=x\gamma\left(h,\overline{h}\right),\\
  \ \ \ f\cdot \gamma\left(h,\overline{h}\right) &= -\frac{1}{2\lambda}\left((h+2)\overline{h}+b\right)\gamma\left(h+2,\overline{h}\right)-\frac{1}{2\lambda}
  \left(\overline{h}^{2}+a\right)\overline{\partial}\left(\gamma\left(h+2,\overline{h}\right)\right),
\end{align*}
\begin{align*}
 \Theta(\lambda,a,b):\ \ \ f\cdot \gamma\left(h,\overline{h}\right)&=2\lambda \overline{\partial}\left(\gamma\left(h+2,\overline{h}\right)\right),\ \ \ \ \ \ \ \ \ \ \ \ \ \ \overline{f}\cdot \gamma\left(h,\overline{h}\right)=\lambda\gamma\left(h+2,\overline{h}\right), \\
  \ \ \ \overline{e}\cdot \gamma\left(h,\overline{h}\right)&=-\frac{1}{4\lambda}\left(\overline{h}^{2}+a\right)\gamma\left(h-2,\overline{h}\right), \ \ \ \ \,x\cdot \gamma\left(h,\overline{h}\right)=x\gamma\left(h,\overline{h}\right),\\
  \ \ \ e\cdot \gamma\left(h,\overline{h}\right) &= -\frac{1}{2\lambda}\left((h-2)\overline{h}+b\right)\gamma\left(h-2,\overline{h}\right)+\frac{1}{2\lambda}
  \left(\overline{h}^{2}+a\right)\overline{\partial}\left(\gamma\left(h-2,\overline{h}\right)\right),
\end{align*}
\begin{align*}
  \Omega\left(\lambda,a,\beta\left(\overline{h}\right)\right):\ \ \ e\cdot \gamma\left(h,\overline{h}\right)&=\left(\frac{\lambda}{2}h+\al\left(\overline{h}\right)\right)\gamma\left(h-2,\overline{h}\right)
  -\lambda\left(\overline{h}+a\right)
  \overline{\partial}\left(\gamma\left(h-2,\overline{h}\right)\right),\\
  \ \ \ f\cdot \gamma\left(h,\overline{h}\right) &= -\left(\frac{1}{2\lambda}h-\beta\left(\overline{h}\right)\right)\gamma\left(h+2,\overline{h}\right)
  -\frac{1}{\lambda}\left(\overline{h}-a\right)\overline{\partial}\left(\gamma\left(h+2,\overline{h}\right)\right),\\
  \ \ \ \overline{e}\cdot \gamma\left(h,\overline{h}\right)&=\frac{\lambda}{2}\left(\overline{h}+a\right)\gamma\left(h-2,\overline{h}\right), \ \ \ x\cdot \gamma\left(h,\overline{h}\right)=x\gamma\left(h,\overline{h}\right),\\
  \ \ \ \overline{f}\cdot \gamma\left(h,\overline{h}\right)&=-\frac{1}{2\lambda}\left(\overline{h}-a\right)\gamma\left(h+2,\overline{h}\right),
\end{align*}
where $x\in \overline{\mathfrak{h}}$ and $\al\left(\overline{h}\right),\beta\left(\overline{h}\right)\in \C\left[\,\overline{h}\,\right]$  with $\al\left(\overline{h}\right),\beta\left(\overline{h}\right)$ satisfying
\begin{align}\label{1}
\left(
\begin{array}{c}
p_{0} \\
p_{1} \\
p_{2} \\
\vdots \\
p_{m} \\
\end{array}
\right)=\lambda^{2}A\left(
                        \begin{array}{c}
                          q_{0} \\
                          q_{1} \\
                          q_{2} \\
                          \vdots \\
                          q_{m} \\
                        \end{array}
                      \right)
                      \ \ \ {\rm with}\ \ A=\left(
                       \begin{array}{ccccc}
                         1 & 2b & 2b^{2} & \cdots & 2b^{m} \\
                         0 & 1 & 2b & \cdots & 2b^{m-1} \\
                         0 & 0 & 1 &\cdots  & 2b^{m-2} \\
                         \vdots & \vdots & \vdots& \ddots & \vdots \\
                         0 & 0 & 0 & \cdots & 1 \\
                       \end{array}
                     \right),\end{align}
where $m={\rm deg}_{\overline{h}}\left(\al\left(\overline{h}\right)\right)={\rm deg}_{\overline{h}}\left(\beta\left(\overline{h}\right)\right)$ and $p_{i},q_{i}$, for $i\in \{0,\ldots,m\}$, are coefficients of $\al\left(\overline{h}\right)$ and $\beta\left(\overline{h}\right)$, respectively.\end{defi}
For later use, we need the following known result,
which gives the conditions for irreducibility and a characterization of isomorphic relations between the $\g$-modules constructed in Definition \ref{defi01}.
\begin{prop}\label{prop-2.2a}\rm{(see \cite{Z2024})} Let $\lambda\in\C^{\times}$, $a,b\in\C$ and
$\beta\left(\overline{h}\right)\in \C\left[\,\overline{h}\,\right]$. Then the following holds:
\begin{itemize}
  \item [\rm(i)] The $\g$-modules $\Gamma\left(\lambda,a,b\right)$ and $\Theta(\lambda,a,b)$
  are simple, while the $\g$-module
  $\Omega\left(\lambda,a,\beta\left(\overline{h}\right)\right)$ is simple
  if and only if $a\neq0$.
  \item [\rm(ii)] $\Gamma\left(\lambda,a,b\right)$, $\Theta(\lambda,a,b)$
  and $\Omega\left(\lambda,a,\beta\left(\overline{h}\right)\right)$ are
  pairwise non-isomorphic for all parameter choices. Moreover,
  \begin{eqnarray*}
   \Gamma(\lambda,a,b)&\cong& \Gamma(\lambda',a',b')\Longleftrightarrow \lambda=\lambda', a=a', b=b', \\
    \Theta(\lambda,a,b)&\cong& \Theta(\lambda',a',b')\Longleftrightarrow  \lambda=\lambda', a=a', b=b',\\
    \Omega\left(\lambda,a,\beta\left(\overline{h}\right)\right)&\cong& \Omega\left(\lambda',a',\beta'\left(\overline{h}\right)\right)\Longleftrightarrow  \lambda=\lambda', a=a',
  \beta\left(\overline{h}\right)=\beta'\left(\overline{h}\right).
  \end{eqnarray*}
\end{itemize}
\end{prop}
For any $\eta,\theta\in \C$, let $I(\eta,\theta)$ be the left ideal of $U(\g)$ generated by the following elements:
$$\{e,\overline{e},h-\theta,\overline{h}-\eta\}.$$
The Verma $\g$-module with highest weight $(\eta,\theta)$ is defined as the quotient module
$$\overline{L}(\eta,\theta)=U(\g)/ I(\eta,\theta).$$
By the PBW theorem, $\overline{L}(\eta,\theta)$ has a basis consisting of all vectors of the form
$$f^{i}\overline{f}^{j}\cdot \overline{v},$$
where $\overline{v}$ is the coset of 1 in $\overline{L}(\eta,\theta)$, and $i,j\in \Z_{+}$. 
Then we have the irreducible highest weight module $L(\eta,\theta)=\overline{L}(\eta,\theta)/J$, where $J$ is the unique maximal proper submodule of $\overline{L}(\eta,\theta)$.

In the following, we will always assume that $\lambda\in \C^{\times}$, $a,b,\eta,\theta\in \C$, $\beta\left(\overline{h}\right)\in \C\left[\,\overline{h}\,\right]$, 
$V(\lambda,a,b_{\b}):= \Gamma(\lambda,a,b),\Theta(\lambda,a,b)$ or $\Omega\left(\lambda,a,\beta\left(\overline{h}\right)\right)$ constructed in Definition \ref{defi01}, 
where $b_{\b}=b$, when $V=\Gamma,\Theta$, and $b_{\b}=\beta\left(\overline{h}\right)$ when $V=\Omega$. 
Let $L(\eta,\theta)$ be an irreducible highest weight $\g$-module.
\section{Irreducibility of the tensor product modules}
In this section, the irreducibility of the tensor product $\g$-module $V(\lambda,a,b_{\b})\otimes L(\eta,\theta)$ is given, where $a\in \C^{\times}$ and $b_{\b}=\beta\left(\overline{h}\right)$ when $V=\Omega$.
\begin{theo}The tensor product module $V(\lambda,a,b_{\b})\otimes L(\eta,\theta)$ is irreducible provided that $a\neq 0$ when $V=\Omega$.
\end{theo}
\begin{proof}Let $M$ be a nonzero submodule of $V(\lambda,a,b_{\b})\otimes L(\eta,\theta)$. We want to prove that $M=V(\lambda,a,b_{\b})\otimes L(\eta,\theta)$. Note that for any $w\in L(\eta,\theta)$, there is a positive integer $K(w)$ such that $\overline{e}^{m}\cdot w=e^{m}\cdot w=0$ for all $m\geq K(w)$. Take any nonzero element $$x=\sum_{i=0}^{r}g_{i}\left(\overline{h}\right)h^{i}\otimes w_{i}\in M,$$ where all $g_{i}\left(\overline{h}\right)\in \C\left[\,\overline{h}\,\right]$, $w_{i}\in L(\eta,\theta)$, $g_{r}\left(\overline{h}\right)\neq0$, $w_{r}\neq0$ and $r\in \Z_{+}$ is minimal. Now we have the following claims.
\begin{clai} \label{claim1}$r=0$.
\end{clai}
 Here we only consider the case $V=\Gamma$, since the other two cases can be proved similarly. Let $K=\mbox {max}\{K(w_{i})\,|\, i=0,\cdots r\}$. It follows that
 \begin{equation*}
   \lambda^{-m}\overline{e}^{m}\cdot x=\sum_{i=0}^{r}g_{i}\left(\overline{h}\right)(h-2m)^{i}\otimes w_{i}\in M ,\ \ \ \mbox {for\ any}\ m\geq K.
 \end{equation*}
 Rewrite the right-hand side of the above formula as
 \begin{equation*}
   \sum_{i=0}^{r}m^{i}x_{i}\in M ,\ \ \ \mbox {for\ any}\ m\geq K,
 \end{equation*}
 where all $x_{i}\in V(\lambda,a,b_{\b})\otimes L(\eta,\theta)$ are independent of $m$. Taking $m=K,K+1\cdots, K+r$, we see that the coefficient matrix of $x_{i}$ is a Vandermonde matrix. So $x_{i}\in M$ for all $i=0,\cdots r$. In particular, $x_{r}=(-2)^{r}g_{r}\left(\overline{h}\right)\otimes w_{r}\in M$. Thus, $r$ must be zero by its minimality.
 \begin{clai} $M=V(\lambda,a,b_{\b})\otimes L(\eta,\theta)$.
\end{clai}
 By Claim \ref{claim1}, $g_{0}\left(\overline{h}\right)\otimes w_{0}\in M$. Thus, $\C\left[\,h,\overline{h}\,\right]g_{0}\left(\overline{h}\right)\otimes w_{0}\subseteq M$. Now we consider the following cases.

 \noindent{\bf Case 1:} $V=\Gamma$. Fix this $w_{0}$ and let $$P_{\Gamma}=\{g\left(h,\overline{h}\right)\in \C\left[\,h,\overline{h}\,\right]\,|\, g\left(h,\overline{h}\right)\otimes w_{0}\in M \}.$$
  Note that $P_{\Gamma}$ is not an empty set as $g_{0}\left(\overline{h}\right)\otimes w_{0}\in P_{\Gamma}$. 
  Suppose that deg$_{\overline{h}}g_{0}\left(\overline{h}\right)=j$.
  Without loss of generality, we can assume $j\geq K(w_{0})$, 
  applying $e^{j}$ on $g_{0}\left(\overline{h}\right)\otimes w_{0}$, we get
 \begin{align*}
   e^{j}\cdot \left(g_{0}\left(\overline{h}\right)\otimes w_{0}\right)=e^{j}\cdot g_{0}\left(\overline{h}\right)\otimes w_{0}=(-2\lambda)^{j}1\otimes w_{0}\in M,
 \end{align*}
which implies that $1\otimes w_{0}\in M$. Hence, $\C\left[\,h,\overline{h}\,\right]\subseteq P_{\Gamma}$, which in turn force that $\C\left[\,h,\overline{h}\,\right]=P_{\Gamma}$. Let $$P'_{\Gamma}:=\{u\in L(\eta,\theta)\,|\,\C\left[\,h,\overline{h}\,\right]\otimes u \subseteq M\}.$$ The previous argument implies that $P'_{\Gamma}\neq \emptyset$. Since $P'_{\Gamma}$ is a $\g$-submodule of $L(\eta,\theta)$, we have $M=V(\lambda,a,b_{\b})\otimes L(\eta,\theta)$ by the irreducibility of $L(\eta,\theta)$, as desired.

\noindent{\bf Case 2:} $V=\Theta$. Applying some linear combination of $e^{i}\overline{e}^{j}$ with $i,j\in \Z_{+}$ on $g_{0}\left(\overline{h}\right)\otimes w_{0}$, then by Definition \ref{defi01} and the fact that  $\C\left[\,h,\overline{h}\,\right]g_{0}\left(\overline{h}\right)\otimes w_{0}\subseteq M$, we obtain $$g_{0}\left(\overline{h}\right)\otimes v\in M,$$ where $v$ is the highest weight vector of $L(\eta,\theta)$. Moreover, $\C\left[\,h,\overline{h}\,\right]g_{0}\left(\overline{h}\right)\otimes v\subseteq M$.
Write $w_{0}=\sum_{i=0}^{s}\sum_{j=0}^{l}a_{i,j}f^{i}\overline{f}^{j}\cdot v$ with $a_{i,j}\in \C$, $a_{s,l}\neq 0$, and denote $\sum_{i=0}^{s}\sum_{j=0}^{l}a_{i,j}f^{i}\overline{f}^{j}$ by $x_{f,\overline{f}}$. Then applying it on $g_{0}\left(\overline{h}\right)\otimes v$, we get
\begin{align*}
  x_{f,\overline{f}}\cdot \left(g_{0}\left(\overline{h}\right)\otimes v\right) &= x_{f,\overline{f}}\cdot g_{0}\left(\overline{h}\right)\otimes v+ g_{0}\left(\overline{h}\right)\otimes x_{f,\overline{f}}\cdot v \\
 & =\SUM{j=0}{l}a_{0,j}\overline{f}^{j}\cdot g_{0}\left(\overline{h}\right)\otimes v+\cdots +\SUM{j=0}{l}a_{s,j}f^{s}\overline{f}^{j}\cdot g_{0}\left(\overline{h}\right)\otimes v+g_{0}\left(\overline{h}\right)\otimes w_{0}\\
 &=\SUM{j=0}{l}a_{0,j}\lambda^{j}\cdot g_{0}\left(\overline{h}\right)\otimes v+\cdots +\SUM{j=0}{l}2^{s}\lambda^{j+s}a_{s,j} \overline{\partial}^{s}\left(g_{0}\left(\overline{h}\right)\right)\otimes v+g_{0}\left(\overline{h}\right)\otimes w_{0}\in M,
\end{align*}
where $\overline{\ptl}:=\frac{\ptl}{\ptl \overline{h}}$ is the partial derivative
with respect to $\overline{h}$ on $\C\left[\,h,\overline{h}\,\right]$. Thus, there is a polynomial $g_{1}\left(\overline{h}\right)\in \C\left[\,\overline{h}\,\right]$ with deg$_{\overline{h}}g_{1}\left(\overline{h}\right)<$deg$_{\overline{h}}g_{0}\left(\overline{h}\right)$ such that $g_{1}\left(\overline{h}\right)\otimes v\in M$. With a finite procedure, we get $1\otimes v\in M$. Then in a similar way in Case 1, we obtain $M=V(\lambda,a,b_{\b})\otimes L(\eta,\theta)$.

\noindent{\bf Case 3:} $V=\Omega$. We can have $M=V(\lambda,a,b_{\b})\otimes L(\eta,\theta)$ in a similar proof of Case 2 if $a\neq 0$. Now let $a=0$. It is clear that $\C\left[\,h,\overline{h}\,\right]\overline{h}\otimes L(\eta,\theta)$ is a submodule of $V(\lambda,a,b_{\b})\otimes L(\eta,\theta)$.

Therefore, by above three cases, we complete the proof.
\end{proof}
\section{Isomorphism classes of the tensor product modules}
In this section, we classify the isomorphism classes of the tensor product $\g$-modules $V(\lambda,a,b_{\b})\otimes L(\eta,\theta)$, where $a\in \C^{\times}$ and $b_{\b}=\beta\left(\overline{h}\right)$ when $V=\Omega$. 
\begin{theo}Let $\lambda_{i}\in \C^{\times}$, $a_{i},b_{i},\eta_{i},\theta_{i}\in \C$, $\beta_{i}\left(\overline{h}\right)\in \C\left[\,\overline{h}\,\right]$ and $L(\eta_{i},\theta_{i})$ be irreducible weight modules, where $i=1,2$. Let $V(\lambda_{i},a_{i},b_{\beta_{i}})$ be the $\g$-modules constructed in Definition \ref{defi01} such that $b_{\beta_{i}}=b_{i}$ when $V=\Gamma, \Theta$ or $a_{i}\in \C^{\times}$ and $b_{\beta_{i}}=\beta_{i}\left(\overline{h}\right)$ when $V=\Omega$ for $i=1,2$. Then $V(\lambda_{1},a_{1},b_{\beta_{1}})\otimes L(\eta_{1},\theta_{1})$ is isomorphic to  $V(\lambda_{2},a_{2},b_{\beta_{2}})\otimes L(\eta_{2},\theta_{2})$ as $\g$-modules if and only if $V(\lambda_{1},a_{1},b_{\beta_{1}})\cong V(\lambda_{2},a_{2},b_{\beta_{2}})$ and $L(\eta_{1},\theta_{1})\cong L(\eta_{2},\theta_{2})$.
\end{theo}
\begin{proof}The sufficiency is obvious and it suffices to show the necessity. 

\noindent{\bf Case 1:} $V=\Gamma$. Let $\psi_{1}: \Gamma(\lambda_{1},a_{1},b_{1})\otimes L(\eta_{1},\theta_{1})\longrightarrow \Gamma(\lambda_{2},a_{2},b_{2})\otimes L(\eta_{2},\theta_{2})$ be a $\g$-module isomorphism. Take a nonzero element $w\in L(\eta_{1},\theta_{1})$ and suppose
$$\psi_{1}(1\otimes w)=\sum_{i=0}^{r}g_{i}\left(\overline{h}\right)h^{i}\otimes w_{i},$$ 
where $g_{i}\left(\overline{h}\right)\in \C\left[\,\overline{h}\,\right]$, $w_{i}\in L(\eta_{2},\theta_{2})$ with $g_{r}\left(\overline{h}\right)\neq 0$ and $w_{r}\neq 0$. 
 There is a positive integer $K$ such that $e^{m}\cdot w= e^{m}\cdot w_{i}=\overline{e}^{m}\cdot w= \overline{e}^{m}\cdot w_{i}=0$ for all $m \geq K$ and $0\leq i \leq r$. Let $m_{1}, m_{2}$ be two different integers such that $m_{1}, m_{2}\geq K$. Then it follows that 
\begin{equation}\label{eq4.1}
  \left(\lambda_{1}^{-m_{1}}\overline{e}^{m_{1}}-\lambda_{1}^{-m_{2}}\overline{e}^{m_{2}}\right)\cdot (1\otimes w)=0.
\end{equation}
Thus applying $\psi_{1}$ on both sides of equation (\ref{eq4.1}), we get
\begin{align}\label{eq4.2}
	0 &= \left(\lambda_{1}^{-m_{1}}\overline{e}^{m_{1}}-\lambda_{1}^{-m_{2}}\overline{e}^{m_{2}}\right)\cdot \psi_{1}(1\otimes w) \nonumber \\
	&= \sum_{i=0}^{r}\left(\lambda_{1}^{-m_{1}}\overline{e}^{m_{1}}-\lambda_{1}^{-m_{2}}\overline{e}^{m_{2}}\right)\cdot \left(g_{i}\left(\overline{h}\right)h^{i}\otimes w_{i}\right) \nonumber \\
	&= \sum_{i=0}^{r}\left(\left(\frac{\lambda_{2}}{\lambda_{1}}\right)^{m_{1}}g_{i}\left(\overline{h}\right)(h-2m_{1})^{i}\otimes w_{i}-\left(\frac{\lambda_{2}}{\lambda_{1}}\right)^{m_{2}}g_{i}\left(\overline{h}\right)(h-2m_{2})^{i}\otimes w_{i}\right).
\end{align}
Comparing the coefficients of $h^{r}$ in (\ref{eq4.2}), we can see that $$\left(\left(\frac{\lambda_{2}}{\lambda_{1}}\right)^{m_{1}}-\left(\frac{\lambda_{2}}{\lambda_{1}}\right)^{m_{2}}\right)
g_{r}\left(\overline{h}\right)\otimes w_{r}=0,$$ which implies that $\lambda_{1}=\lambda_{2}$. Denote $\lambda=\lambda_{1}=\lambda_{2}$. Now observing the equation (\ref{eq4.2}), we find that $r=0$, i.e., $\psi_{1}(1\otimes w)=g_{0}\left(\overline{h}\right)\otimes w_{0}$. 
\\
For any $m\geq K$, the equations
\begin{align*}
	\psi_{1}\left(\left(-4\overline{e}^{m}\overline{f}-\overline{e}^{m-1}\overline{h}^2\right)\left(1\otimes w\right)\right)&=\left(   -4\overline{e}^{m}\overline{f}-\overline{e}^{m-1}\overline{h}^2\right)\psi_1\left(1\otimes w\right), \\
	\psi_{1}\left(\overline{e}^{m}h\left(1\otimes w\right)\right)&=\overline{e}^{m}h\psi_1\left(1\otimes w\right),\\
	\psi_{1}\left(\overline{e}^{m}\overline{h}\left(1\otimes w\right)\right)&=\overline{e}^{m}\overline{h}\psi_1\left(1\otimes w\right),
\end{align*}
are, respectively, equivalent to
\begin{align*}
	a_{2}\lambda^{m-1}g_{0}\left(\overline{h}\right)\otimes w_{0}&=a_{1}\lambda^{m-1}g_{0}\left(\overline{h}\right)\otimes w_{0}, \\
	\lambda^m\left(\psi_1\left(h-2m\right)\otimes w\right)&=\lambda^m\left(h-2m\right)g_{0}\left(\overline{h}\right)\otimes w_{0}, \\
	\lambda^m\psi_1\left(\overline{h}\otimes w\right)&=\lambda^m\overline{h}g_{0}\left(\overline{h}\right)\otimes w_{0},
\end{align*}
we can see that $a_{1}=a_{2}$, denote $a=a_{1}=a_{2}$,
\begin{align}
	\psi_{1}\left(h\otimes w\right)&=hg_{0}\left(\overline{h}\right)\otimes w_{0}, \\
	\psi_{1}\left(\overline{h}\otimes w\right)&=\overline{h}g_{0}\left(\overline{h}\right)\otimes w_{0}, \\
	\psi_{1}\left(h\overline{h}\otimes w\right)&=h\overline{h}g_{0}\left(\overline{h}\right)\otimes w_{0}. 
\end{align}
From (4.3)-(4.5) and 
\begin{equation*}
	\psi_{1}\left(\overline{e}^{m}f\left(1\otimes w\right)\right)=\overline{e}^{m}f\psi_{1}\left(1\otimes w\right), \ \ \ \mbox {where}\ m\geq K, 
\end{equation*}
we deduce that 
\begin{equation*}
	\ b_{1}g_{0}\left(\overline{h}\right)\otimes w_{0}=\ b_{2}g_{0}\left(\overline{h}\right)\otimes w_{0}+\left(\overline{h}^{2}+a\right)\overline{\partial} g_{0}\left(\overline{h}\right)\otimes w_{0}. 
\end{equation*}
We can see that $b_{1}=b_{2},   \ g_{0}\left(\overline{h}\right) \in \C^{\times}$, we assume that $g_{0}\left(\overline{h}\right)=1$, thus $\psi_{1}\left(1\otimes w\right)=1\otimes w_{0}$. Hence $\Gamma(\lambda_{1},a_{1},b_{1})\cong \Gamma(\lambda_{2},a_{2},b_{2})$ by Proposition \ref{prop-2.2a}. Thus there exists a linear map $\tau_{1}:L(\eta_{1},\theta_{1})\longrightarrow L(\eta_{2},\theta_{2})$ such that 
\begin{equation}
	\psi_{1}\left(1\otimes w\right)=1\otimes\tau_{1}\left(w\right),\quad   \forall \,w\in L(\eta_{1},\theta_{1}).
\end{equation}
According to (4.3)-(4.6), we obtain 
\begin{equation*}
	\psi_{1}\left(\left(\g\cdot1\right)\otimes w\right)=\left(\g\cdot1\right)\otimes\tau_{1}\left(w\right),\quad\forall \,w\in L(\eta_{1},\theta_{1}). 
\end{equation*}
This, in conjunction with
\begin{equation*}
	\psi_{1}\left(\g\cdot\left(1\otimes w\right)\right)=\g\cdot\psi_{1}\left(1\otimes w\right),\quad\forall \,w\in L(\eta_{1},\theta_{1}), 
\end{equation*}
leads to
\begin{equation*}
	\psi_{1}\left(1\otimes \g\cdot w\right)=1\otimes\tau_{1}\left(\g\cdot w\right),\quad\forall \,w\in L(\eta_{1},\theta_{1}). 
\end{equation*}
Therefore, 
\begin{equation*}
	\tau_{1}\left(\g\cdot w\right)=\g\cdot\tau_{1}\left(w\right).
\end{equation*}
Thus, $\tau_{1}$ is a non-zero $\g$-module homomorphism. 
Given that $L(\eta_{1},\theta_{1})$ and $L(\eta_{2},\theta_{2})$ are simple $\g$-modules, $\tau_{1}$ is a $\g$-module isomorphism. \\
\noindent{\bf Case 2:}  $V=\Theta$. Let $\psi_{2}: \Theta(\lambda_{1},a_{1},b_{1})\otimes L(\eta_{1},\theta_{1})\longrightarrow \Theta(\lambda_{2},a_{2},b_{2})\otimes L(\eta_{2},\theta_{2})$ be a $\g$-modules isomorphism. Take a nonzero element $w\in L(\eta_{1},\theta_{1})$ and suppose $$\psi_{2}(1\otimes w)=\sum_{i=0}^{r}g_{i}\left(\overline{h}\right)h^{i}\otimes w_{i},$$ where $g_{i}\left(\overline{h}\right)\in \C\left[\,\overline{h}\,\right]$, $w_{i}\in L(\eta_{2},\theta_{2})$ with $g_{r}\left(\overline{h}\right)\neq 0$ and $w_{r}\neq 0$. There is a positive integer $K$ such that $e^{m}\cdot w= e^{m}\cdot w_{i}=\overline{e}^{m}\cdot w= \overline{e}^{m}\cdot w_{i}=0$ for all $m \geq K$ and $0\leq i \leq r$. Then it follows that
\begin{equation}
	\left(\left(-4\lambda_{1}\right)^{m}\overline{e}^{m}-\left(-4\lambda_{1}\right)^{m-1}\overline{e}^{m-1}\overline{h}^{2}-\left(-4\lambda_{1}\right)a_{1}\overline{e}^{m-1}\right)\left(1\otimes w\right)=0. 
\end{equation}
Applying $\psi_{2}$ on both sides of equation (4.7), we get
\begin{align}
	0 &= \left((-4\lambda_1)^m\,\overline{e}^m-(-4\lambda_1)^{m-1}\,\overline{e}^{m-1}\overline{h}^2-(-4\lambda_1)\,a_1\overline{e}^{m-1}\right)\psi_2\,(1\otimes w) \nonumber \\
	\begin{split}
		&= \sum_{i=0}^r \Bigg( \left(\frac{\lambda_1}{\lambda_2}\right)^m \left(\overline{h}^2+a_2\right)^m (h-2m)^i \\
		&\quad - \left(\frac{\lambda_1}{\lambda_2}\right)^{m-1} \left(\overline{h}^2+a_2\right)^{m-1} \left(\overline{h}^2+a_1\right) (h-2(m-1))^i \Bigg) g_i\left(\overline{h}\right) \otimes w_i. \label{eq,4.8}
	\end{split} 
\end{align}
Comparing the coefficients of $h^{r}$ in (4.8), we can see that
\begin{equation*}
	\left(\left(\frac{\lambda_1}{\lambda_2}\right)^m \left(\overline{h}^2+a_2\right)^m-\left(\frac{\lambda_1}{\lambda_2}\right)^{m-1}\left(\overline{h}^2+a_2\right)^{m-1}\left(\overline{h}^2+a_1\right)\right)g_{r}(\overline{h})\otimes w_{r}=0. 
\end{equation*}
Since $g_{r}\neq0, w_{r}\neq0$, we can conclude $\lambda_{1}=\lambda_{2}, a_{1}=a_{2}$. Denote $\lambda=\lambda_{1}=\lambda_{2}, a=a_{1}=a_{2}$. Rewriting the equation (4.8), we have 
\begin{equation*}
	\SUM{i=0}{r}\left(\left(h-2m\right)^{i}-\left(h-2(m-1)\right)^{i}\right)\left(\overline{h}^{2}+a\right)^{m}g_{i}\left(\overline{h}\right)\otimes w_{i}=0. 
\end{equation*}
If $r\geq1$, the highest degree term in $h$ is
\begin{equation*}
	-2rh^{r-1}\left(\overline{h}^{2}+a\right)^{m}g_{r}\left(\overline{h}\right)\otimes w_{r}=0. 
\end{equation*} 
This leads to a contradiction, hence we conclude that $r=0, \psi_{2}\left(1\otimes w\right)=g_{0}(\overline{h})\otimes w_{0}$. For any $m\geq K$, the equations
\begin{align*}
	\psi_{2}\left(\overline{e}^{m}f\left(1\otimes w\right)\right)&=\overline{e}^{m}f\psi_{2}\left(1\otimes w\right), \\
	\psi_{2}\left(\overline{e}^{m}\left(1\otimes w\right)\right)&=\overline{e}^{m}\psi_{2}\left(1\otimes w\right), \\
	\psi_{2}\left(\overline{e}^{m}\overline{h}\left(1\otimes w\right)\right)&=\overline{e}^{m}\overline{h}\psi_{2}\left(1\otimes w\right), \\
	\psi_{2}\left(\overline{e}^{m}h\left(1\otimes w\right)\right)&=\overline{e}^{m}h\psi_{2}\left(1\otimes w\right), \\
	\psi_{2}\left(\overline{e}^{m}h\overline{h}\left(1\otimes w\right)\right)&=\overline{e}^{m}h\overline{h}\psi_{2}\left(1\otimes w\right),
\end{align*}
are, respectively, equivalent to
\begin{align}
	\left(\overline{h}^{2}+a\right)^{m}\overline{\partial} g_{0}\left(\overline{h}\right)\otimes w_{0}&=0,\label{eq,4.9}\\
	\psi_{2}\left(\left(\overline{h}^{2}+a\right)^{m}\otimes w\right)&=\left(\overline{h}^{2}+a\right)^{m}\otimes w_{0},\label{eq,4.10}\\
	\psi_{2}\left(\left(\overline{h}^{2}+a\right)^{m}\overline{h}\otimes w\right)&=\left(\overline{h}^{2}+a\right)^{m}\overline{h}\otimes w_{0},\label{eq,4.11}\\
	\psi_{2}\left(\left(\overline{h}^{2}+a\right)^{m}h\otimes w\right)&=\left(\overline{h}^{2}+a\right)^{m}h\otimes w_{0},\label{eq,4.12}\\
	\psi_{2}\left(\left(\overline{h}^{2}+a\right)^{m}\overline{h}h\otimes w\right)&=\left(\overline{h}^{2}+a\right)^{m}\overline{h}h\otimes w_{0}.  \label{eq,4.13}
\end{align}
From \eqref{eq,4.9}, we can get that $g_{0}\left(\overline{h}\right)\in \C^{\times}$, we assume that $g_{0}\left(\overline{h}\right)=1$, thus $\psi_{2}\left(1\otimes w\right)=1\otimes w_{0}$.
From \eqref{eq,4.10}-\eqref{eq,4.13} and 
\begin{equation*}
	\psi_{2}\left(\overline{e}^{m}e\cdot\left(1\otimes w\right)\right)=\overline{e}^{m}e\cdot\psi_{2}\left(1\otimes w\right),\ \ \ \mbox {where}\ m\geq K, 
\end{equation*}
we can see that
\begin{equation*}
 	b_{1}\otimes w_{0}=b_{2}\otimes w_{0}, 
\end{equation*}
thus $b_{1}=b_{2}$. Hence $\Theta(\lambda_{1},a_{1},b_{1})\cong \Theta(\lambda_{2},a_{2},b_{2})$ by Proposition \ref{prop-2.2a}. Thus there exists a linear map $\tau_{2}:L(\eta_{1},\theta_{1})\longrightarrow L(\eta_{2},\theta_{2})$ such that
\begin{equation}
	\psi_{2}\left(1\otimes w\right)=1\otimes\tau_{2}\left(w\right),\quad   \forall \,w\in L(\eta_{1},\theta_{1}). 
\end{equation} \label{eq,4.14}
In $\Theta\left(\lambda,a,b\right)$, if $a\neq0$, we obtain 
\begin{equation*}
	\SUM{k=0}{m}\binom{m}{k}a^{m-k}\left(\psi_{2}\left(\overline{h}^{2k}\otimes w\right)-\overline{h}^{2k}\otimes w_{0}\right)=0,\quad k\in\Z_{+},
\end{equation*}
by \eqref{eq,4.10}. Regard it as a polynomial in $a$, we can see that $\psi_{2}\left(\overline{h}^{2k}\otimes w\right)=\overline{h}^{2k}\otimes w_{0}$. When $k=1$, we can get 
\begin{equation}
	\psi_{2}\left(\overline{h}^{2}\otimes w\right)=\overline{h}^{2}\otimes w_{0}. \label{eq,4.15}
\end{equation}
In equations \eqref{eq,4.11}-\eqref{eq,4.13}, by the similar approach, we can derive 
\begin{align}
	\psi_{2}\left(\overline{h}\otimes w\right)&=\overline{h}\otimes w_{0}, \label{eq,4.16}\\
	\psi_{2}\left(h\otimes w\right)&=h\otimes w_{0}, \label{eq,4.17}\\
	\psi_{2}\left(h\overline{h}\otimes w\right)&=h\overline{h}\otimes w_{0}. \label{eq,4.18} 
\end{align}
Suppose $a=0$, the equation \eqref{eq,4.11} becomes $$	\psi_{2}\left(\left(\overline{h}^{2}-z+z\right)^{m}\overline{h}\otimes w\right)=\left(\overline{h}^{2}-z+z\right)^{m}\overline{h}\otimes w_{0},\quad z\in\C^{\times}.$$ Using a method similar to that for the case  $a\neq0$, we again derive equations \eqref{eq,4.15}-\eqref{eq,4.18}. From these equations, it follows that the isomorphism between $L\left(\eta_{1},\theta_{1}\right)$ and $ L\left(\eta_{2},\theta_{2}\right)$ can be established by a proof similar to that in Case 1. We therefore conclude that $L\left(\eta_{1},\theta_{1}\right)\cong L\left(\eta_{2},\theta_{2}\right)$.\\
\noindent{\bf Case 3:}  $V=\Omega$. A proof similar to Case 2 shows that if $\Omega\left(\lambda_{1},a_{1},\beta_{1}\left(\overline{h}\right)\right)\otimes L\left(\eta_{1},\theta_{1}\right)\cong\Omega\left(\lambda_{2},a_{2},\beta_{2}\left(\overline{h}\right)\right)\otimes L\left(\eta_{2},\theta_{2}\right)$, then  $\Omega\left(\lambda_{1},a_{1},\beta_{1}\left(\overline{h}\right)\right)\cong\Omega\left(\lambda_{2},a_{2},\beta_{2}\left(\overline{h}\right)\right)$ and $L\left(\eta_{1},\theta_{1}\right)\cong L\left(\eta_{2},\theta_{2}\right)$.

We have thus exhausted all cases, which completes the proof.
\end{proof}
\section{Comparison of tensor product modules with known non-weight modules}
In this section, we will compare the tensor product modules is different from all known irreducible non-weight $\g$-modules. 
Let $\underline{\mu}=\left(\mu_{1},\mu_{2}\right)\in\C^{2}$. Assume that $J_{\underline{\mu}}$ is the left ideal of $U\left(\overline{\mathfrak{n}}_{+}\right)$ generated by $\left\{e-\mu_{1},\overline{e}-\mu_{2}\right\}$. Denote $N_{\underline{\mu}}= U\left(\overline{\mathfrak{n}}_{+}\right)/J_{\underline{\mu}}$. Ind$\left(N_{\underline{\mu}}\right)= U\left(\g\right)\otimes_{U\left(\overline{\mathfrak{n}}_{+}\right)} N_{\underline{\mu}}$ is a universal Whittaker module. 
For any $r>deg_{h}g\left(h,\overline{h}\right)$, we denote
\begin{align*}
	w^{(r)}&=\sum_{i=0}^{r}\binom{r}{i}\left(-1\right)^{r-i}\lambda^{-i}\overline{e}^{i},\quad \text{when} \ V=\Gamma,\\
	w^{(r)}&=\sum_{i=0}^{r}\binom{r}{i}\left(-1\right)^{r-i}\lambda^{-i}\overline{f}^{i},\quad \text{when} \ V=\Theta,\\
	w^{(r)}&=\sum_{i=0}^{r}\binom{r}{i}\left(-1\right)^{r-i}\left(\frac{a^{2}}{4}\right)^{-i}\left(\overline{e}\overline{f}+\frac{1}{4}\overline{h}^{2}\right)^{i},\quad \text{when} \ V=\Omega. 
\end{align*} 
\begin{lemm}\label{lem:5.1}
	Let $V\left(\lambda,a,b_{\beta}\right)$ and $L\left(\eta,\theta\right)$ be the $\g$-modules defined as in Section \ref{sec:preliminaries}. Then, the following statements hold:
	\begin{itemize}
		\item [\rm(i)] For any $g\left(h,\overline{h}\right)\in V\left(\lambda,a,b_{\beta}\right)$, we have 
		\begin{eqnarray*}
			w^{(r)}\cdot g\left(h,\overline{h}\right)=0,\quad \forall \,r>deg_{h}g\left(h,\overline{h}\right). 
		\end{eqnarray*}
		 \item[\rm(ii)] For any $ r>deg_{h}g\left(h,\overline{h}\right)$ and $0\neq g\left(h,\overline{h}\right)\in V\left(\lambda,a,b_{\beta}\right)$ when $V=\Omega, \lambda\neq a$, there exists $v\in L\left(\eta,\theta\right)$ such that
		 \begin{equation*}
		 	w^{(r)}\cdot\left(g\left(h,\overline{h}\right)\otimes v\right)\neq0. 
		 \end{equation*}
	\end{itemize}
	\begin{proof}
		(i) For any $g\left(\overline{h}\right)\in\C\left[\,\overline{h}\,\right]$ and $j\in\Z_{+}$.\\
		\noindent{\bf Case 1:} $V=\Gamma$.
		\begin{align*}
			w^{(r)}\cdot\left(h^{j}g\left(\overline{h}\right)\right)
			&=\sum_{i=0}^{r}\binom{r}{i}\left(-1\right)^{r-i}\lambda^{-i}\overline{e}^{i}\cdot\left(h^{j}g\left(\overline{h}\right)\right)\\
			&=\sum_{i=0}^{r}\binom{r}{i}\left(-1\right)^{r-i}\left(h-2i\right)^{j}g\left(\overline{h}\right)\\
			&=\sum_{i=0}^{r}\binom{r}{i}\left(-1\right)^{r-i}\left(a_{j}i^{j}+a_{j-1}i^{j-1}+\dots+a_{1}i+a_{0}\right)g\left(\overline{h}\right), 
		\end{align*}
		regarding $a_{0},a_{1}\dots a_{j}$ as polynomials in $h$. Using the identity equation 
		\begin{equation*}
			\sum_{i=0}^{r}\binom{r}{i}\left(-1\right)^{r-i}i^{j}=0,\quad\forall \ j,r\in\Z_{+},\ j<r,
		\end{equation*}
		we can get $w^{(r)}\cdot\left(h^{j}g\left(\overline{h}\right)\right)=0$.  \\
		\noindent{\bf Case 2:} $V=\Theta$.
		\begin{align*}
			w^{(r)}\cdot\left(h^{j}g\left(\overline{h}\right)\right)
			&=\sum_{i=0}^{r}\binom{r}{i}\left(-1\right)^{r-i}\lambda^{-i}\overline{f}^{i}\cdot\left(h^{j}g\left(\overline{h}\right)\right)\\
			&=\sum_{i=0}^{r}\binom{r}{i}\left(-1\right)^{r-i}\left(h+2i\right)^{j}g\left(\overline{h}\right)\\
			&=0.
		\end{align*}
		\noindent{\bf Case 3:} $V=\Omega$.
		\begin{align*}
			w^{(r)}\cdot\left(h^{j}g\left(\overline{h}\right)\right)
			&=\sum_{i=0}^{r}\binom{r}{i}\left(-1\right)^{r-i}\left(\frac{a^{2}}{4}\right)^{-i}\left(\overline{e}\overline{f}+\frac{1}{4}\overline{h}^{2}\right)^{i}\cdot\left(h^{j}g\left(\overline{h}\right)\right)\\
			&=\sum_{i=0}^{r}\binom{r}{i}\left(-1\right)^{r-i}h^{j}g\left(\overline{h}\right)\\
			&=0. 
		\end{align*}
		Thus $w^{(r)}\cdot g\left(h,\overline{h}\right)=0$.\\
		(ii) For any $ r>deg_{h}g\left(h,\overline{h}\right)$, take $v$ to be the highest weight vector of $L\left(\eta,\theta\right)$. \\
		\noindent{\bf Case 1:} $V=\Gamma$.
		\begin{align*}
			w^{(r)}\cdot \left(g\left(h,\overline{h}\right)\otimes v\right)&=\sum_{i=0}^{r}\binom{r}{i}\left(-1\right)^{r-i}\lambda^{-i}\overline{e}^{i}\cdot \left(g\left(h,\overline{h}\right)\otimes v\right)\\
			&=\sum_{i=0}^{r}\binom{r}{i}\left(-1\right)^{r-i}\lambda^{-i}g\left(h,\overline{h}\right)\otimes\overline{e}^{i}\cdot v\\
			&=\left(-1\right)^{r}g\left(h,\overline{h}\right)\otimes v\\
			&\neq0.
		\end{align*}
		\noindent{\bf Case 2:} $V=\Theta$.
		\begin{align*}
			w^{(r)}\cdot \left(g\left(h,\overline{h}\right)\otimes v\right)&=\sum_{i=0}^{r}\binom{r}{i}\left(-1\right)^{r-i}\lambda^{-i}\overline{f}^{i}\cdot \left(g\left(h,\overline{h}\right)\otimes v\right)\\
			&=\sum_{i=0}^{r}\binom{r}{i}\left(-1\right)^{r-i}\lambda^{-i}g\left(h,\overline{h}\right)\otimes\overline{f}^{i}\cdot v\\
			&\neq0.
		\end{align*}
		\noindent{\bf Case 3:} $V=\Omega, \lambda\neq a$.
		\begin{align*}
			w^{(r)}\cdot \left(g\left(h,\overline{h}\right)\otimes v\right)&=\sum_{i=0}^{r}\binom{r}{i}\left(-1\right)^{r-i}\left(\frac{a^{2}}{4}\right)^{-i}\left(\overline{e}\overline{f}+\frac{1}{4}\overline{h}^{2}\right)^{i}\cdot \left(g\left(h,\overline{h}\right)\otimes v\right)\\
			&=\sum_{i=0}^{r}\binom{r}{i}\left(-1\right)^{r-i}\left(\frac{a^{2}}{4}\right)^{-i}g\left(h,\overline{h}\right)\otimes \left(\overline{e}\overline{f}+\frac{1}{4}\overline{h}^{2}\right)^{i}\cdot v\\
			&=\sum_{i=0}^{r}\binom{r}{i}\left(-1\right)^{r-i}\left(\frac{\lambda^{2}}{a^{2}}\right)^{i}g\left(h,\overline{h}\right)\otimes v\\
			&=\left(\frac{\lambda^{2}}{a^{2}}-1\right)^{r}g\left(h,\overline{h}\right)\otimes v\\
			&\neq0.  
		\end{align*}
	     Thus, (ii) is proved, and we complete the proof. 
	\end{proof}
\end{lemm}	
\begin{prop}
	Let $\eta,\theta\in\C$ and not both be zero. Then $V\left(\lambda,a,b_{\beta}\right)\otimes L\left(\eta,\theta\right)$ is a new non-weight $\g$-module. 
\end{prop}
\begin{proof}
	It must be shown that the tensor product module $V\left(\lambda,a,b_{\beta}\right)\otimes L\left(\eta,\theta\right)$ is isomorphic to neither a Whittaker module nor a free $U\left(\overline{\mathfrak{h}}\right)$-free modules of rank one. Let $W$ be a Whittaker module over $\g$ which is isomorphic to a quotient module of $\text{Ind}\left(N_{\underline{\mu}}\right)$ for some $\underline{\mu}=\left(\mu_{1},\mu_{2}\right)\in \C^{2}$. Suppose $V\left(\lambda,a,b_{\beta}\right)\otimes L\left(\eta,\theta\right)\cong W$ and let $\phi:V\left(\lambda,a,b_{\beta}\right)\otimes L\left(\eta,\theta\right)\longrightarrow W$ be the isomorphism of $\g$-module. Then $\phi\left(1\otimes v\right)\neq0$, where $v$ is the highest weight vector of $L\left(\eta,\theta\right)$. Now consider the following cases:\\
	\noindent{\bf Case 1:} $V=\Gamma$.  
	Applying $e$ on $\phi\left(1\otimes v\right)$, we obtain
	\begin{equation*}
		\mu_{1}\phi\left(1\otimes v\right)=e\cdot\phi\left(1\otimes v\right)=\phi\left(e\cdot1\otimes v\right)=0, 
	\end{equation*}
	thus $\mu_{1}=0$. Applying $e$ on $\phi\left(\overline{h}\otimes v\right)$, we obtain 
	\begin{equation*}
		0=e\cdot\phi\left(\overline{h}\otimes v\right)=\phi\left(e\cdot\overline{h}\otimes v\right)=-2\lambda\phi\left(1\otimes v\right),
	\end{equation*}
	thus $\phi\left(1\otimes v\right)=0$. We obtain a contradiction against the assumption.\\
	\noindent{\bf Case 2:} $V=\Theta$. 
	Applying $\overline{e},\overline{h}^{2}$ on $\phi\left(1\otimes v\right)$ respectively, we obtain
	\begin{gather*}
		\mu_{2}\phi\left(1\otimes v\right)=\overline{e}\cdot\phi\left(1\otimes v\right)=\phi\left(\overline{e}\cdot1\otimes v\right)=-\frac{1}{4\lambda}\phi\left(\overline{h}^{2}\otimes v\right)-\frac{a}{4\lambda}\phi\left(1\otimes v\right),\\
		\overline{h}^{2}\cdot\phi\left(1\otimes v\right)=\phi\left(\overline{h}^{2}\otimes v\right)+\eta^{2}\phi\left(1\otimes v\right).
	\end{gather*}
	Combining the two equations above, we obtain
	\begin{equation*}
		\left(\overline{h}^{2}-\eta^{2}+a+4\lambda\mu_{2}\right)\cdot\phi\left(1\otimes v\right)=0.
	\end{equation*}
	This yields a contradiction to $\phi\left(1\otimes v\right)\neq0$.\\
	\noindent{\bf Case 3:} $V=\Omega$. 
	Appyling $\overline{e},\overline{h}$ on $\phi\left(1\otimes v\right)$ respectively, we can see that \\
	\begin{gather*}
		\mu_{2}\phi\left(1\otimes v\right)=\overline{e}\cdot\phi\left(q\otimes v\right)=\phi\left(\overline{e}\cdot 1\otimes v\right)=\frac{\lambda}{2}\phi\left(\overline{h}\otimes v\right)+\frac{\lambda a}{2}\phi\left(1\otimes v\right),\\
		\overline{h}\cdot\phi\left(1\otimes v\right)=\phi\left(\overline{h}\otimes v\right)+\eta\phi\left(1\otimes v\right). 
	\end{gather*}
	Combining the two equations above, we have
	\begin{equation*}
		\left(\overline{h}-\eta+a-\frac{2\mu_{2}}{\lambda}\right)\cdot\phi\left(1\otimes v\right)=0. 
	\end{equation*}
	Thus we arrive at a contradiction. Hence $V\left(\lambda,a,b_{\beta}\right)\otimes L\left(\eta,\theta\right)\ncong W$. 
	It follows from Lemma \ref{lem:5.1} that $V\left(\lambda,a,b_{\beta}\right)\otimes L\left(\eta,\theta\right)\ncong V\left(\lambda,a,b_{\beta}\right)$, completing the proof. 
\end{proof}
\section{Realization of tensor product modules as induced modules}
In this section, we consider the tensor product modules $V\left(\lambda,a,b_{\beta}\right)\otimes\overline{L}\left(\eta,\theta\right)$ as induced modules from modules over certain subalgebras of $\g$. And the conditions for these modules to be reducible are also determined from modules over certain subalgebras of $\g$. Let 
\begin{align*}
	\mathfrak{b}&=\text{span}_{\C}\left\{\overline{e},e,\overline{h}\right\},\quad \text{when} \ V=\Gamma ,\\
	\mathfrak{b}&=\text{span}_{\C}\left\{\overline{e},\overline{h}\right\},\quad \text{when} \ V=\Theta \ \text{or} \ \Omega ,
\end{align*}
be the subalgebra of $\g$.
\begin{defi}\label{def:6.1}
	Let $\C\left[\,\overline{h}\,\right]$ denote the polynomial algebra with respect to the variable $\overline{h}$. For $\lambda\in\C^{\times}$, $a,b,\eta,\theta\in \C$, $\beta\left(\overline{h}\right)\in \C\left[\,\overline{h}\,\right]$, where $b_{\b}=b$, when $V=\Gamma,\Theta$, and $b_{\b}=\beta\left(\overline{h}\right)$ when $V=\Omega$. We define the action of $\mathfrak{b}$ on $\C\left[\,\overline{h}\,\right]$ as follows:
	\begin{flalign*}
		\ \ \ \ \ \ \ \ \ \ \ \ \ \ \ \ \ \ \ \ \ \ \ \ \ \ \ \ \ \ \C\left[\,\overline{h}\,\right]^{\Gamma}:\ \ \ \overline{e}\circ g\left(\overline{h}\right)&=\lambda g\left(\overline{h}\right),&\\
		e\circ g\left(\overline{h}\right)&=-2\lambda\overline{\partial}\left(g\left(\overline{h}\right)\right),&\\
		\overline{h}\circ g\left(\overline{h}\right)&=\left(\overline{h}+\eta\right)g\left(\overline{h}\right),&
	\end{flalign*}
	\begin{flalign*}
		\ \ \ \ \ \ \ \ \ \ \ \ \ \ \ \ \ \ \ \ \ \ \ \ \ \ \ \ \ \ \C\left[\,\overline{h}\,\right]^{\Theta}:\ \ \ \overline{e}\circ g\left(\overline{h}\right)&=-\frac{1}{4\lambda}\left(\overline{h}^{2}+a\right)g\left(\overline{h}\right),&\\
		\overline{h}\circ g\left(\overline{h}\right)&=\left(\overline{h}+\eta\right)g\left(\overline{h}\right),&
	\end{flalign*}
	\begin{flalign*}
		\ \ \ \ \ \ \ \ \ \ \ \ \ \ \ \ \ \ \ \ \ \ \ \ \ \ \ \ \ \ \C\left[\,\overline{h}\,\right]^{\Omega}:\ \ \ \overline{e}\circ g\left(\overline{h}\right)&= \frac{\lambda}{2}\left(\overline{h}+a\right)g\left(\overline{h}\right),&\\
		\overline{h}\circ g\left(\overline{h}\right)&=\left(\overline{h}+\eta\right)g\left(\overline{h}\right).&
	\end{flalign*}
\end{defi}
\begin{prop}
	Keep notations as above in Definition \ref{def:6.1}, then $\C\left[\,\overline{h}\,\right]^{\Gamma}, \C\left[\,\overline{h}\,\right]^{\Theta}$ and $\C\left[\,\overline{h}\,\right]^{\Omega}$ are $\mathfrak{b}$-modules under the actions given in Definition \ref{def:6.1}. 
\end{prop}
\begin{proof}
  This claim can be proved straightforwardly, we omit it.
\end{proof}
\begin{remark}
Let $\lambda\in \C^{\times}$, $a,\eta\in\C$, the $\mathfrak{b}$-module $\C\left[\,\overline{h}\,\right]^{\Gamma}$ is irreducible, the $\mathfrak{b}$-modules $\C\left[\,\overline{h}\,\right]^{\Theta}$ and $\C\left[\,\overline{h}\,\right]^{\Omega}$ are reducible.
\end{remark}
From the Definition \ref{def:6.1}, we obtain the induced $\g$-modules as follows:
\begin{align*}
	\ \ \ \text{Ind}\C\left[\,\overline{h}\,\right]^{\Gamma} &\mathrel{\mathop:}=U\left(\g\right)\otimes_{U\left(\mathfrak{b}\right)}\C\left[\,\overline{h}\,\right]^{\Gamma},\\
	\ \ \ \text{Ind}\C\left[\,\overline{h}\,\right]^{\Theta} &\mathrel{\mathop:}=U\left(\g\right)\otimes_{U\left(\mathfrak{b}\right)}\C\left[\,\overline{h}\,\right]^{\Theta},\\
	\ \ \ \text{Ind}\C\left[\,\overline{h}\,\right]^{\Omega} &\mathrel{\mathop:}=U\left(\g\right)\otimes_{U\left(\mathfrak{b}\right)}\C\left[\,\overline{h}\,\right]^{\Omega}.
\end{align*}
\begin{theo}\label{theo:6.4}
	Keep notations as above. Then as $\g$-modules
	\begin{align*}
		\Gamma\left(\lambda,a,b\right)\otimes\overline{L}\left(\eta,\theta\right)&\cong{\rm
			Ind}\C\left[\,\overline{h}\,\right]^{\Gamma},\\
		\Theta\left(\lambda,a,b\right)\otimes\overline{L}\left(\eta,\theta\right)&\cong{\rm Ind}\C\left[\,\overline{h}\,\right]^{\Theta},\\
		\Omega\left(\lambda,a,\beta\left(\overline{h}\right)\right)\otimes\overline{L}\left(\eta,\theta\right)&\cong{\rm Ind}\C\left[\,\overline{h}\,\right]^{\Omega}.
	\end{align*}
\end{theo}
\begin{proof}
	In the following, we only prove the first case, since a similar argument can be applied to the other two cases. According to the PBW Theorem, we can see that $\Gamma\left(\lambda,a,b\right)\otimes\overline{L}\left(\eta,\theta\right)$ has a basis
	\begin{equation*}
		\mathcal{B}_{1}=\left\{\overline{h}^{i}h^{q}\otimes\left(f^{j}\overline{f}^{k}\cdot \overline{v}\right)\mid i,q,j,k\in\Z_{+}\right\},
	\end{equation*}
	 where $\overline{v}$ is the coset of 1 in $\overline{L}\left(\eta,\theta\right)$. We define a total order $\prec$ on $\mathcal{B}_{1}$ 
	\begin{equation*}
		h^{q}\overline{h}^{i}\otimes\left(f^{j}\overline{f}^{k}\cdot \overline{v}\right)\prec h^{q^{\prime}}\overline{h}^{i^{\prime}}\otimes\left(f^{j^{\prime}}\overline{f}^{k^{\prime}}\cdot \overline{v}\right), 
	\end{equation*}
	if and only if 
	\begin{equation*}
		\left(k,j,i,q\right)<\left(k^{\prime},j^{\prime},i^{\prime},q^{\prime}\right),
	\end{equation*}
	where $\left(a_{1},a_{2},a_{3},a_{4}\right)<\left(b_{1},b_{2},b_{3},b_{4}\right)\Longleftrightarrow\exists\, k>0$ such that $a_{i}=b_{i}$ for all $i<k$ and $a_{k}<b_{k}$. \\
	$\text{Ind}\C\left[\,\overline{h}\,\right]^{\Gamma}$ has a basis
	\begin{equation*}
		\mathcal{B}_{2}=\left\{f^{j}\overline{f}^{k}h^{q}\otimes\overline{h}^{i}\mid j,k,q,i\in\Z_{+}\right\}. 
	\end{equation*}
	We define the linear map: $$\varphi:\ \ \text{Ind}\C\left[\,\overline{h}\,\right]^{\Gamma}\longrightarrow\Gamma\left(\lambda,a,b\right)\otimes\overline{L}\left(\eta,\theta\right)$$ by $$\varphi\left(f^{j}\overline{f}^{k}h^{q}\otimes\overline{h}^{i}\right)=f^{j}\overline{f}^{k}h^{q}\left(\overline{h}^{i}\otimes \overline{v}\right).$$
	We claim that $\varphi$ is a $\g$-module homomorphism. 
	From Definition \ref{def:6.1}, for any $i,q,j,k\in\Z_{+}$ and $x=f^{j}\overline{f}^{k}h^{q},$ we have
	\begin{align*}
		\varphi\left(x\overline{e}\otimes\overline{h}^{i}\right)&=\varphi\left(x\otimes\overline{e}\circ\overline{h}^{i}\right)=x\left(\overline{e}\circ\overline{h}^{i}\otimes \overline{v}\right)=x\overline{e}\left(\overline{h}^{i}\otimes \overline{v}\right),\\
		\varphi\left(xe\otimes\overline{h}^{i}\right)&=\varphi\left(x\otimes  e\circ\overline{h}^{i}\right)=x\left(e\circ\overline{h}^{i}\otimes \overline{v}\right)=xe\left(\overline{h}^{i}\otimes \overline{v}\right),\\
		\varphi\left(x\overline{h}\otimes\overline{h}^{i}\right)&=\varphi\left(x\otimes\overline{h}\circ\overline{h}^{i}\right)=x\left(\overline{h}\circ\overline{h}^{i}\otimes \overline{v}\right)=x\overline{h}\left(\overline{h}^{i}\otimes \overline{v}\right).
	\end{align*}
	From PBW Theorem , we have 
	\begin{equation*}
		zx=zf^{j}\overline{f}^{k}h^{q}=\sum_{j_{a}}X_{j_{a}}x_{j_{a}}+\sum_{j_{b}}Y_{j_{b}}y_{j_{b}}\overline{e}+\sum_{j_{c}}Z_{j_{c}}z_{j_{c}}e+\sum_{j_{d}}U_{j_{d}}u_{j_{d}}\overline{h}, 
	\end{equation*}
	where $z\in\left\{\overline{e},e,\overline{h},h,\overline{f},f \right\}$, $X_{j_{a}},Y_{j_{b}},Z_{j_{c}},U_{j_{d}}\in\C$, $x_{j_{a}}\otimes\overline{h}^{i},y_{j_{b}}\otimes\overline{h}^{i},z_{j_{c}}\otimes\overline{h}^{i},u_{j_{d}}\otimes\overline{h}^{i}\in\mathcal{B}_{2}$. We have
	\begin{align*}
		\varphi\left(zx\otimes\overline{h}^{i}\right)&=\varphi\left(\left(\sum_{j_{a}}X_{j_{a}}x_{j_{a}}+\sum_{j_{b}}Y_{j_{b}}y_{j_{b}}\overline{e}+\sum_{j_{c}}Z_{j_{c}}z_{j_{c}}e+\sum_{j_{d}}U_{j_{d}}u_{j_{d}}\overline{h}\right)\otimes \overline{h}^{i}\right)\\
		&=\sum_{j_{a}}X_{j_{a}}\varphi\left(x_{j_{a}}\otimes\overline{h}^{i}\right)+\sum_{j_{b}}Y_{j_{b}}\varphi\left(y_{j_{b}}\overline{e}\otimes\overline{h}^{i}\right)\\
		&\quad+\sum_{j_{c}}Z_{j_{c}}\varphi\left(z_{j_{c}}e\otimes\overline{h}^{i}\right)+\sum_{j_{d}}U_{j_{d}}\varphi\left(u_{j_{d}}\overline{h}\otimes\overline{h}^{i}\right)\\
		&=\sum_{j_{a}}X_{j_{a}}x_{j_{a}}\left(\overline{h}^{i}\otimes \overline{v}\right)+\sum_{j_{b}}Y_{j_{b}}y_{j_{b}}\overline{e}\left(\overline{h}^{i}\otimes \overline{v}\right)\\
		&\quad+\sum_{j_{c}}Z_{j_{c}}z_{j_{c}}e\left(\overline{h}^{i}\otimes \overline{v}\right)+\sum_{j_{d}}U_{j_{d}}u_{j_{d}}\overline{h}\left(\overline{h}^{i}\otimes \overline{v}\right)\\
		&=zx\left(\overline{h}^{i}\otimes \overline{v}\right)\\
		&=z\varphi\left(x\otimes\overline{h}^{i}\right).
	\end{align*}
	Therefore,  $\varphi$ is a $\g$-module homomorphism.
	We claim that $\varphi$ is a surjection. We shall prove that $\Gamma\left(\lambda,a,b\right)\otimes\overline{L}\left(\eta,\theta\right)\subseteq \text{Im}\left(\varphi\right)$. We know that $\overline{h}^{i}\otimes \overline{v}\in\text{Im}\left(\varphi\right)$ for all $i\in\Z_{+}$. Then, for any $q\in\Z_{+}$, $h^{q}\left(\overline{h}^{i}\otimes \overline{v}\right)\in\text{Im}\left(\varphi\right)$, we obtain $h^{q}\overline{h}^{i}\otimes \overline{v}\in\text{Im}\left(\varphi\right)$. Hence $\Gamma\left(\lambda,a,b\right)\otimes \overline{v}\subseteq\text{Im}\left(\varphi\right)$. Applying $f^{j}\overline{f}^{k}h^{q}$ on $\Gamma\left(\lambda,a,b\right)\otimes \overline{v}$, we have $\Gamma\left(\lambda,a,b\right)\otimes\overline{L}\left(\eta,\theta\right)=\text{Im}\left(\varphi\right)$. Thus, $\varphi$ is a surjection.\\ 
	Next we prove that $\varphi$ is an injection.
	Through computation, we find that 
	\begin{equation*}
		f^{j}\overline{f}^{k}h^{q}\left(\overline{h}^{i}\otimes \overline{v}\right)=h^{q}\overline{h}^{i}\otimes\left(f^{j}\overline{f}^{k}\cdot \overline{v}\right)+\text{lower terms},
	\end{equation*}
	this implies that the set $\left\{f^{j}\overline{f}^{k}h^{q}\left(\overline{h}^{i}\otimes \overline{v}\right)\mid j,k,q,i\in\Z_{+}\right\}$ forms a basis of $\mathcal{B}_{1}$. Therefore, $\varphi$ is injective. The proof is completed. 
\end{proof}
\begin{coro}
	Let $\lambda\in\C^{\times}$, $a,b,\eta,\theta\in\C$, $\beta\left(\overline{h}\right)\in\C\left[\,\overline{h}\,\right]$. Then the following holds.
	\begin{itemize}
		\item [\rm(i)] The induced module ${\rm Ind}\C\left[\,\overline{h}\,\right]^{V}$ for $V\in\left\{\Gamma,\Theta\right\}$ is reducible if and only if $\eta=0$ and $i\left(\theta-2j-i+1\right)=0$ for $j\in\Z_{+},i\in\left\{0,1\right\}$ and $\left(i,j\right)\neq\left(0,0\right)$. 
		\item [\rm(ii)]  The induced module ${\rm Ind}\C\left[\,\overline{h}\,\right]^{\Omega}$ is reducible if and only if $a=0$ or $\eta=0$ and $i\left(\theta-2j-i+1\right)=0$ for $j\in\Z_{+},i\in\left\{0,1\right\}$ and $\left(i,j\right)\neq\left(0,0\right)$. 
	\end{itemize}	 
\end{coro}
\begin{proof}
	It follows from Theorem \ref{theo:6.4} that $\text{Ind}\C\left[\,\overline{h}\,\right]^{V}$ is reducible for $V\in\left\{\Gamma,\Theta,\Omega\right\}$ if and only if $V\left(\lambda,a,b_{\beta}\right)$ or $\overline{L}\left(\eta,\theta\right)$ is reducible. It is suffice to show that the module $\overline{L}\left(\eta,\theta\right)$ is reducible 
	if and only if there exists a nonzero vector $\overline{w}$, distinct from 
	the highest weight vector $\overline{v}$, such that $e \cdot \overline{w} =\overline{e}\cdot \overline{w}= 0$.
	Let $$\overline{w}=f^{i}\overline{f}^{j}\cdot \overline{v},\quad i,j\in\Z_{+},\ \left(i,j\right)\neq\left(0,0\right).$$
	We have 
	\begin{align*}
		\left[\,e,f^{i}\,\right]&=if^{i-1}\left(h-i+1\right),\\
		\left[\,e,\overline{f}^{j}\,\right]&=j\overline{f}^{j-1}\overline{h},\\
		\left[\,\overline{e},f^{i}\,\right]&=if^{i-1}\overline{h}-i\left(i-1\right)\overline{f}f^{i-2},\\
		\left[\,\overline{e},\overline{f}^{j}\,\right]&=0.
	\end{align*}
	By computation, we obtain
	\begin{align*}
		e\cdot \overline{w}&=j\eta f^{i}\overline{f}^{j-1}\cdot \overline{v}+i\left(
		\theta-2j-i+1\right)f^{i-1}\overline{f}^{j}\cdot \overline{v}=0,\\
		\overline{e}\cdot \overline{w}&=i\eta f^{i-1}\overline{f}^{j}\cdot \overline{v}-i\left(i-1\right)f^{i-2}\overline{f}^{j+1}\cdot \overline{v}=0.
	\end{align*}
	Consequently, we have $$\eta=0, \,i\left(\theta-2j-i+1\right)=0,\, \text{for}\ j\in\Z_{+},\, i\in\left\{0,1\right\},\, \left(i,j\right)\neq\left(0,0\right).$$ 
Combining this with Proposition \ref{prop-2.2a}, we complete the proof.
\end{proof}
\section*{Acknowledgments}
This work was supported by the Zhejiang Provincial Natural Science Foundation of China, 
Grant No. LQN25A010023.

 \end{document}